\documentclass[12pt]{amsart}

\usepackage{amsmath}
\usepackage{amssymb}
\usepackage{amsfonts}
\usepackage{amsthm}
\usepackage{enumerate}
\usepackage{hyperref}
\usepackage{color}

\theoremstyle{plain}
\newtheorem{thm}{Theorem}[section]

\newtheorem{prop}[thm]{Proposition}

\newtheorem{cor}[thm]{Corollary}

\theoremstyle{definition}

\newtheorem{rem}[thm]{Remark}

\newtheorem{dfns-rems}[thm]{Definitions and Remarks}
\newtheorem{notas-rems}[thm]{Notations and Remarks}
\newtheorem{exmps-rems}[thm]{Examples and Remarks}

\begin{document}


\title[Stanley depth and symbolic powers]{Stanley depth and symbolic powers of monomial ideals}


\author[S. A. Seyed Fakhari]{S. A. Seyed Fakhari}

\address{S. A. Seyed Fakhari, School of Mathematics, Institute for Research
in Fundamental Sciences (IPM), P.O. Box 19395-5746, Tehran, Iran.}

\email{fakhari@ipm.ir}

\urladdr{http://math.ipm.ac.ir/fakhari/}


\begin{abstract}
The aim of this paper is to study the Stanley depth of symbolic powers of a squarefree monomial ideal. We prove that for every squarefree monomial ideal $I$ and every pair of integers $k, s\geq 1$, the inequalities ${\rm sdepth} (S/I^{(ks)}) \leq {\rm sdepth} (S/I^{(s)})$ and ${\rm sdepth} (I^{(ks)}) \leq {\rm sdepth} (I^{(s)})$ hold. If moreover $I$ is unmixed of height $d$, then we show that for every integer $k\geq1$, ${\rm sdepth}(I^{(k+d)})\leq {\rm sdepth}(I^{{(k)}})$ and ${\rm sdepth}(S/I^{(k+d)})\leq {\rm sdepth}(S/I^{{(k)}})$. Finally, we consider the limit behavior of the Stanley depth of symbolic powers of a squarefree monomial ideal. We also introduce a method for comparing the Stanley depth of factors of monomial ideals.
\end{abstract}






\thanks{This research was in part supported by a grant from IPM (No. 92130422)}


\maketitle


\section{Introduction} \label{sec1}

Let $\mathbb{K}$ be a field and $S=\mathbb{K}[x_1,\dots,x_n]$ be the
polynomial ring in $n$ variables over the field $\mathbb{K}$. Let $M$ be a nonzero
finitely generated $\mathbb{Z}^n$-graded $S$-module. Let $u\in M$ be a
homogeneous element and $Z\subseteq \{x_1,\dots,x_n\}$. The $\mathbb
{K}$-subspace $u\mathbb{K}[Z]$ generated by all elements $uv$ with $v\in
\mathbb{K}[Z]$ is called a {\it Stanley space} of dimension $|Z|$, if it is
a free $\mathbb{K}[\mathbb{Z}]$-module. Here, as usual, $|Z|$ denotes the
number of elements of $Z$. A decomposition $\mathcal{D}$ of $M$ as a finite
direct sum of Stanley spaces is called a {\it Stanley decomposition} of
$M$. The minimum dimension of a Stanley space in $\mathcal{D}$ is called
{\it Stanley depth} of $\mathcal{D}$ and is denoted by ${\rm
sdepth}(\mathcal {D})$. The quantity $${\rm sdepth}(M):=\max\big\{{\rm
sdepth}(\mathcal{D})\mid \mathcal{D}\ {\rm is\ a\ Stanley\ decomposition\
of}\ M\big\}$$ is called {\it Stanley depth} of $M$. Stanley \cite{s}
conjectured that $${\rm depth}(M) \leq {\rm sdepth}(M)$$ for all
$\mathbb{Z}^n$-graded $S$-modules $M$. As a convention, we set ${\rm sdepth}(M)=0$ , when $M$ is the zero module. For a reader friendly introduction
to Stanley depth, we refer the reader to \cite{psty}.

In this paper, we introduce a method for comparing the Stanley depth of factors of monomial ideals (see Theorem \ref{main}). We show that our method implies the known results regarding the Stanley depth of  radical, integral closure and colon of monomial ideals (see Propositions \ref{rad}, \ref{int1}, \ref{int2}, \ref{colon}).

In Section \ref{sec3}, we apply our method for studying the Stanley depth of symbolic powers of monomial ideals. We show that for every pair of integers $k, s\geq 1$ the Stanley depth of the $k$th symbolic power of a squarefree monomial ideal $I$ is an upper bound for the Stanley depth of the $(ks)$th symbolic power of $I$ (see Theorem \ref{csymbol}). If moreover $I$ is unmixed of height $d$, then we show that for every integer $k\geq 1$, the Stanley depth of the $k$th symbolic power of $I$ is an upper bound for the Stanley depth of the $(k+d)$th symbolic power of $I$ (see Theorem \ref{unmix}). Finally, in Theorem \ref{better} we show that the limit behavior of the Stanley depth of unmixed squarefree monomial  ideals can be very interesting.


\section{A comparison tool for the Stanley depth} \label{sec2}

The following theorem is the main result of this section. Using this result, we deduce some known results regarding the Stanley depth of  radical, integral closure and colon of monomial ideals. We should mention that in the following theorem by ${\rm Mon}(S)$, we mean the set of all monomials in the polynomial ring $S$.

\begin{thm} \label{main}
Let $I_2\subseteq I_1$ and $J_2\subseteq J_1$ be monomial ideals in $S$. Assume that there exists a function $\phi: {\rm Mon}(S)\rightarrow{\rm Mon}(S)$, such that the following conditions are satisfied.
\begin{itemize}
\item[(i)] For every monomial $u\in {\rm Mon}(S)$, $u\in I_1$ if and only if $\phi(u)\in J_1$.
\item[(ii)] For every monomial $u\in {\rm Mon}(S)$, $u\in I_2$ if and only if $\phi(u)\in J_2$.
\item[(iii)] For every Stanley space $u\mathbb{K}[Z]$ and every monomial $v\in {\rm Mon}(S)$, $v\in u\mathbb{K}[Z]$ if and only if $\phi(v)\in \phi(u)\mathbb{K}[Z]$.
\end{itemize}
Then
$${\rm sdepth} (I_1/I_2) \geq {\rm sdepth} (J_1/J_2).$$
\end{thm}

\begin{proof}
Consider a Stanley decomposition $$\mathcal{D} : J_1/J_2=\bigoplus_{i=1}^m t_i \mathbb{K}[Z_i]$$ of $J_1/J_2$, such that ${\rm sdepth}(\mathcal{D})={\rm sdepth} (J_1/J_2)$. By assumptions, for every monomial $u\in I_1\setminus I_2$, we have $$\phi(u)\in J_1\setminus J_2.$$Thus for each monomial $u\in I_1\setminus I_2$, we define $Z_u:=Z_i$ and $t_u:=t_i$, where $i\in\{1, \ldots, m\}$ is the uniquely determined index, such that $\phi(u)\in t_i \mathbb{K}[Z_i]$. It is clear that $$I_1\setminus I_2\subseteq \sum u\mathbb{K}[Z_u],$$where the sum is taken over all monomials $u\in I_1\setminus I_2$. For the converse inclusion note that for every $u\in I_1\setminus I_2$ and every monomial $h\in \mathbb{K}[Z_u]$, clearly we have $uh\in I_1$. By the choice of $t_u$ and $Z_u$, we conclude $\phi(u)\in t_u \mathbb{K}[Z_u]$ and therefore by (iii),  $$\phi(uh)\in \phi(u) \mathbb{K}[Z_u] \subseteq t_u \mathbb{K}[Z_u].$$This implies that $\phi(uh)\notin J_2$ and it follows from (ii) that $uh\notin I_2$. Thus $$I_1/I_2= \sum u\mathbb{K}[Z_u],$$ where the sum is taken over all monomials $u\in I_1\setminus I_2$.

Now for every $1\leq i \leq m$, let
$$U_i=\{u\in I_1\setminus I_2: Z_u=Z_i \ {\rm and} \ t_u=t_i\}.$$
Without lose of generality we may assume that $U_i\neq \emptyset$ for every $1\leq i \leq l$ and $U_i=\emptyset$ for every $l+1\leq i\leq m$. Note that $$I_1/I_2= \sum_{i=1}^l \sum u\mathbb{K}[Z_i],$$ where the second sum is taken over all monomials $u\in U_i$. For every $1\leq i \leq l$, let $u_i$ be the greatest common divisor of elements of $U_i$. We claim that for every $1\leq i \leq l$, $u_i\in U_i$.

{\it Proof of claim.} It is enough to show that $\phi(u_i)\in t_i\mathbb{K}[Z_i]$. This, together with (i) and (ii) implies that $u_i\in I_1\setminus I_2$ and $Z_{u_i}=Z_i$ and $t_{u_i}=t_i$ and hence $u_i\in U_i$. So assume that $t_i$ does not divide $\phi(u_i)$. Then there exists $1\leq j \leq n$, such that ${\rm deg}_{x_j}(\phi(u_i)) < {\rm deg}_{x_j}(t_i)$, where for every monomial $v\in S$, ${\rm deg}_{x_j}(v)$ denotes the degree of $v$ with respect to the variable $x_j$. Also by the choice of $u_i$, there exists a monomial $u\in U_i$, such that ${\rm deg}_{x_j}(u)= {\rm deg}_{x_j}(u_i)$. We conclude that $$u\in u_i\mathbb{K}[x_1, \ldots, x_{j-1}, x_{j+1}, \ldots, x_n],$$and hence by (iii)$$\phi(u)\in \phi(u_i)\mathbb{K}[x_1, \ldots, x_{j-1}, x_{j+1}, \ldots, x_n].$$ This shows that $${\rm deg}_{x_j}(\phi(u))={\rm deg}_{x_j}(\phi(u_i))< {\rm deg}_{x_j}(t_i).$$ It follows that $t_i$ does not divide $\phi(u)$, which is a contradiction, since $\phi(u)\in t_i \mathbb{K}[Z_i]$. Hence $t_i$ divides $\phi(u_i)$. On the other hand, since $u_i$ divides every monomial $u\in U_i$, (iii) implies that for every monomial $u\in U_i$, $\phi(u_i)$ divides $\phi(u)$. Note that by the definition of $U_i$, for every for every monomial $u\in U_i$, $\phi(u)\in t_i\mathbb{K}[Z_i]$. It follows that $$\phi(u_i)\in t_i\mathbb{K}[Z_i]$$and this completes the proof of our claim.

Our claim implies that for every $1\leq i \leq l$, we have$$u_i\mathbb{K}[Z_i]\subseteq \sum_{u\in U_i} u\mathbb{K}[Z_i].$$On the other hand (iii) implies that, for every monomial $u\in U_i$, $\phi(u_i)$ divides $\phi(u)$. Since $$\phi(u_i)\in t_i\mathbb{K}[Z_i] \ \ \ \ and \ \ \ \ \phi(u)\in t_i\mathbb{K}[Z_i],$$we conclude that $$\phi(u)\in \phi(u_i)\mathbb{K}[Z_i]$$and it follows from (iii) that $$u\in u_i\mathbb{K}[Z_i]$$and thus$$u_i\mathbb{K}[Z_i]= \sum_{u\in U_i} u\mathbb{K}[Z_i].$$Therefore $$I_1/I_2= \sum_{i=1}^l u_i\mathbb{K}[Z_i].$$

Next we prove that for every $1\leq i,j \leq l$ with $i\neq j$, the summands $u_i\mathbb{K}[Z_i]$ and $u_j\mathbb{K}[Z_j]$ intersect trivially. By contradiction, let $v$ be a monomial in $u_i\mathbb{K}[Z_i]\cap u_j\mathbb{K}[Z_j]$. Then there exist $h_i\in \mathbb{K}[Z_i]$ and $h_j\in \mathbb{K}[Z_j]$, such that $u_ih_i=v=u_jh_j$. Therefore $\phi(u_ih_i)=\phi(v)=\phi(u_jh_j)$. But $u_i\in U_i$ and hence $\phi(u_i)\in t_i \mathbb{K}[Z_i]$, which by (iii) implies that $$\phi(u_ih_i)\in \phi(u_i)\mathbb{K}[Z_i]\subseteq t_i \mathbb{K}[Z_i].$$Similarly $\phi(u_jh_j)\in t_j \mathbb{K}[Z_j]$. Thus $$\phi(v)\in t_i \mathbb{K}[Z_i]\cap t_j \mathbb{K}[Z_j],$$ which is a contradiction, because $\bigoplus_{i=1}^m t_i \mathbb{K}[Z_i]$ is a Stanley decomposition of $J_1/J_2$. Therefore  $$I_1/I_2=\bigoplus_{i=1}^l u_i \mathbb{K}[Z_i]$$ is a Stanley decomposition of $I_1/I_2$ which proves that$${\rm sdepth} (I_1/I_2)\geq \min_{i=1}^l|Z_i|\geq{\rm sdepth} (J_1/J_2).$$
\end{proof}

Using Theorem \ref{main}, we are able to deduce many known results regarding the Stanley depth of factors of monomial ideals. For example, it is known that the Stanley depth of the radical of a monomial ideal $I$ is an upper bound for the Stanley depth of $I$. In the following proposition we show that this result follows from Theorem \ref{main}.

\begin{prop} \label{rad}
{\rm (}See \cite{a, i'}{\rm )} Let $J\subseteq I$ be monomial ideals in $S$. Then $${\rm sdepth}(I/J)\leq {\rm sdepth}(\sqrt{I}/\sqrt{J}).$$
\end{prop}
\begin{proof}
Let $G(\sqrt{I})=\{u_1, \ldots, u_s\}$ be the minimal set of monomial generators of $\sqrt{I}$. For every $1\leq i \leq s$, there exists an integer $k_i\geq 1$ such that $u_i^{k_i}\in I$. Let $k_I={\rm lcm}(k_1, \ldots,k_s)$ be the least common multiple of $k_1, \ldots, k_s$. Now for every $1\leq i \leq s$, we have $u_i^{k_I}\in I$ and this implies that $u^{k_I}\in I$, for every monomial $u\in \sqrt{I}$. It follows that for every monomial $u\in S$, we have $u\in \sqrt{I}$ if and only if $u^{k_I}\in I$. Similarly there exists an integer $k_J$, such that for every monomial $u\in S$, $u\in \sqrt{J}$ if and only if $u^{k_J}\in J$. Let $k={\rm lcm}(k_I, k_J)$ be the least common multiple of $k_I$ and $k_J$. For every monomial $u\in S$, we define $\phi(u)=u^k$. It is clear that $\phi$ satisfies the hypothesis of Theorem \ref{main}. Hence it follows from that theorem that $${\rm sdepth}(I/J)\leq {\rm sdepth}(\sqrt{I}/\sqrt{J}).$$
\end{proof}

Let $I\subset S$ be an arbitrary ideal. An element $f \in S$ is
{\it integral} over $I$, if there exists an equation
$$f^k + c_1f^{k-1}+ \ldots + c_{k-1}f + c_k = 0 {\rm \ \ \ \ with} \ c_i\in I^i.$$
The set of elements $\overline{I}$ in $S$ which are integral over $I$ is the {\it integral closure}
of $I$. It is known that the integral closure of a monomial ideal $I\subset S$ is a monomial ideal
generated by all monomials $u \in S$ for which there exists an integer $k$ such that
$u^k\in I^k$ (see \cite[Theorem 1.4.2]{hh'}).

Let $I$ be a monomial ideal in $S$ and let $k\geq 1$ be a fixed integer. Then for every monomial $u\in S$, we have  $u\in \overline{I}$ if and only if $u^s\in I^s$, for some $s\geq 1$ if and only if $u^{ks'}\in I^{ks'}$, for some $s'\geq 1$ if and only if $u^k\in \overline{I^k}$. This shows that by setting $\phi(u)=u^k$ in Theorem \ref{main} we obtain the following result from \cite{s1}. We should mention that the method which is used in the proof of Theorem \ref{main} is essentially a generalization of one which is used in \cite{s1}.

\begin{prop} \label{int1}
{\rm (}\cite[Theorem 2.1]{s1}{\rm )} Let $J\subseteq I$ be two monomial ideals in $S$. Then for every integer $k\geq 1$
$${\rm sdepth} (\overline{I^k}/\overline{J^k}) \leq {\rm sdepth} (\overline{I}/\overline{J}).$$
\end{prop}

Similarly, using Theorem \ref{main} we can deduce the following result from \cite{s1}.

\begin{prop} \label{int2}
{\rm (}\cite[Theorem 2.8]{s1}{\rm )} Let $I_2\subseteq I_1$ be two monomial ideals in $S$. Then there  exists an integer $k\geq 1$, such that for every $s\geq 1$
$${\rm sdepth} (I_1^{sk}/I_2^{sk}) \leq {\rm sdepth} (\overline{I_1}/\overline{I_2}).$$
\end{prop}
\begin{proof}
Note that by Remark \cite[Remark 1.1]{s1}, there exist integers  $k_1, k_2\geq 1$, such that for every monomial $u\in S$, we have $u^{k_1}\in I_1^{k_1}$ (resp. $u^{k_2}\in I_2^{k_2}$) if and only if $u\in \overline{I_1}$ (resp. $u\in \overline{I_2}$). Let $k={\rm lcm}(k_1,k_2)$ be the least common multiple of $k_1$ and $k_2$. Then for every monomial $u\in S$, we have $u^k\in I_1^k$ (resp. $u^k\in I_2^k$) if and only if $u\in \overline{I_1}$ (resp. $u\in \overline{I_2}$). Hence for every monomial $u\in S$ and every $s\geq 1$, we have $u^{sk}\in I_1^{sk}$ (resp. $u^{sk}\in I_2^{sk}$) if and only if $u\in \overline{I_1}$ (resp. $u\in \overline{I_2}$). Set $\phi(u)=u^{sk}$, for every monomial $u\in S$ and every $s\geq 1$. Now the assertion follows from Theorem \ref{main}
\end{proof}

Let $I$ be a monomial ideal in $S$ and $v\in S$ be a monomial. It can be easily seen that $(I:v)$ is a monomial ideal. Popescu \cite{p} proves that ${\rm sdepth}(I:v)\geq {\rm sdepth}(I)$. On the other hand, Cimpoeas \cite{c} proves that ${\rm sdepth}(S/(I:v))\geq {\rm sdepth}(S/I)$. Using Theorem \ref{main}, we prove a generalization of these results.

\begin{prop} \label{colon}
Let $J\subseteq I$ be monomial ideals in $S$ and $v\in S$ be a monomial. Then $${\rm sdepth}(I/J)\leq {\rm sdepth}((I:v)/(J:v)).$$
\end{prop}
\begin{proof}
It is just enough to use Theorem \ref{main} by setting $\phi(u)=vu$, for every monomial $u\in S$
\end{proof}


\section{Stanley depth of symbolic powers} \label{sec3}

Let $I$ be a squarefree monomial ideal in $S$ and suppose that $I$ has the irredundant
primary decomposition $$I=\frak{p}_1\cap\ldots\cap\frak{p}_r,$$ where every
$\frak{p}_i$ is an ideal of $S$ generated by a subset of the variables of
$S$. Let $k$ be a positive integer. The $k$th {\it symbolic power} of $I$,
denoted by $I^{(k)}$, is defined to be $$I^{(k)}=\frak{p}_1^k\cap\ldots\cap
\frak{p}_r^k.$$As a convention, we define the $k$th symbolic power of $S$ to be equal to $S$, for every $k\geq 1$.

We now use Theorem \ref{main} to prove a new result. Indeed, we use Theorem \ref{main} to compare the Stanley depth of symbolic powers of squarefree monomial ideals.

\begin{thm} \label{symbol}
Let $J\subseteq I$ be squarefree monomial ideals in $S$. Then for every pair of integers $k,s\geq 1$
$${\rm sdepth} (I^{(ks)}/J^{(ks)}) \leq {\rm sdepth} (I^{(s)}/J^{(s)}).$$
\end{thm}

\begin{proof}
Suppose that $I=\cap_{i=1}^r\frak{p}_i$ is the irredundant primary decomposition of $I$ and let $u\in S$ be a monomial. Then $u\in I^{(s)}$ if and only if for every $1\leq i\leq r$ $$\sum_{x_j\in P_i}{\rm deg}_{x_j}u\geq s$$ if and only if $$\sum_{x_j\in P_i}{\rm deg}_{x_j}u^k\geq sk$$ if and only if $u^k\in I^{(sk)}$. By a similar argument, $u \in J^{(s)}$ if and only if $u^k\in J^{(sk)}$. Thus for proving our assertion, it is enough to use Theorem \ref{main}, by setting $\phi(u)=u^k$, for every monomial $u\in S$.
\end{proof}

The following corollary is an immediate consequence of Theorem \ref{symbol}.

\begin{cor} \label{csymbol}
Let $I$ be a squarefree monomial ideals in $S$. Then for every pair of integers $k,s\geq 1$, the inequalities
$${\rm sdepth} (S/I^{(ks)}) \leq {\rm sdepth} (S/I^{(s)})$$and $${\rm sdepth} (I^{(ks)}) \leq {\rm sdepth} (I^{(s)})$$hold.
\end{cor}

\begin{rem} \label{phi}
Let $t\geq 1$ be a fixed integer. Also let $I$ be a squarefree monomial ideal in $S$ and suppose that
$I=\cap_{i=1}^r\frak{p}_i$ is the irredundant primary decomposition of $I$.
Assume that $A\subseteq \{x_1, \ldots, x_n\}$ is a subset of variables
 of $S$, such that $$|\frak{p}_i\cap A|=t,$$for every $1\leq i \leq r$.
 We set $v=\Pi_{x_i\in A}x_i$. It is clear that for every integer $k\geq 1$ and every integer $1\leq i \leq r$, a monomial $u\in {\rm Mon(S)}$ belongs to $\frak{p}_i^k$ if and only if $uv$ belongs to $\frak{p}_i^{k+t}$. This implies that for every integer $k\geq 1$, a monomial $u\in {\rm Mon(S)}$ belongs to $I^{(k)}$ if and only if $uv$ belongs to $I^{(k+t)}$. This shows $$(I^{(k+t)}:v)=I^{(k)}$$ and thus Proposition \ref{colon} implies that $${\rm sdepth}(I^{(k+t)})\leq {\rm sdepth}(I^{{(k)}})$$and  $${\rm sdepth}(S/I^{(k+t)})\leq {\rm sdepth}(S/I^{{(k)}}).$$In particular, we conclude the following result.
\end{rem}

\begin{prop} \label{one}
Let $I$ be a squarefree monomial ideal in $S$ and suppose there exists a subset $A\subseteq \{x_1, \ldots, x_n\}$ of variables
 of $S$, such that for every prime ideal $\frak{p}\in {\rm Ass}(S/I)$, $$|\frak{p}\cap A|=1.$$Then for every integer $k\geq 1$, the inequalities $${\rm sdepth}(I^{(k+1)})\leq {\rm sdepth}(I^{{(k)}})$$and $${\rm sdepth}(S/I^{(k+1)})\leq {\rm sdepth}(S/I^{{(k)}})$$hold.
\end{prop}

As an example of ideals which satisfy the assumptions of Proposition \ref{one}, we consider the cover ideal of bipartite graphs. Let $G$ be a graph with vertex set $V(G)=\{v_1, \ldots, v_n\}$ and edge set $E(G)$. A subset
$C\subseteq V(G)$ is a {\it minimal vertex cover} of $G$ if, first, every
edge of $G$ is incident with a vertex in $C$ and, second, there is no
proper subset of $C$ with the first property. For a graph $G$ the {\it cover ideal} of $G$ is defined by $$J_G=\bigcap_{\{v_i
,v_j\}\in E(G)}\langle x_i,x_j\rangle.$$ For instance, unmixed squarefree
monomial ideals of height two are just cover ideals of graphs. The
name cover ideal comes from the fact that $J_G$ is generated by squarefree
monomials $x_{i_1}, \ldots, x_{i_r}$ with $\{v_{i_1}, \ldots, v_{i_r}\}$ is
a minimal vertex cover of $G$. A graph $G$ is {\it bipartite}
if there exists a partition $V(G)=U\cup W$ with $U\cap W=\varnothing$ such
that each edge of $G$ is of the form $\{v_i,v_j\}$ with $v_i\in U$ and
$v_j\in W$.

\begin{cor} \label{cover}
Let $G$ be a bipartite graph and $J_G$ be the cover ideal of $G$. Then for every integer $k\geq 1$, the inequalities $${\rm sdepth}(J_G^{(k+1)})\leq {\rm sdepth}(J_G^{(k)})$$and $${\rm sdepth}(S/J_G^{(k+1)})\leq {\rm sdepth}(S/J_G^{(k)})$$hold.
\end{cor}
\begin{proof}
Let $V(G)=U\cup W$ be a partition for the vertex set of $G$. Note that $${\rm Ass}(S/J_G)=\big\{\langle x_i,x_j\rangle:\{v_i
,v_j\}\in E(G) \big\}.$$Thus for every $\frak{p}\in {\rm Ass}(S/J_G)$, we have $|\frak{p}\cap A|=1$, where $$A=\{x_i: v_i\in U\}.$$Now Proposition \ref{one} completes the proof of the assertion.
\end{proof}
It is known \cite[Theorem 5.1]{hht} that for a bipartite graph $G$ with cover ideal $J_G$, we have $J_G^{(k)}=J_G^k$, for every integer $k\geq 1$. Therefore we conclude the following result from  Corollary \ref{cover}.

\begin{cor}
Let $G$ be a bipartite graph and $J_G$ be the cover ideal of $G$. Then for every integer $k\geq 1$, the inequalities $${\rm sdepth}(J_G^{k+1})\leq {\rm sdepth}(J_G^k)$$and $${\rm sdepth}(S/J_G^{k+1})\leq {\rm sdepth}(S/J_G^k)$$hold.
\end{cor}

Let $G$ be a non-bipartite graph and let $J_G$ be its cover ideal. We do not know whether the inequalities $${\rm sdepth}(J_G^{(k+1)})\leq {\rm sdepth}(J_G^{(k)})$$and $${\rm sdepth}(S/J_G^{(k+1)})\leq {\rm sdepth}(S/J_G^{(k)})$$hold for every integer $k\geq 1$. However, we will see in Corollary \ref{nonbi} that we always have the following inequalities.$${\rm sdepth}(J_G^{(k+2)})\leq {\rm sdepth}(J_G^{(k)}) \ \ \ \ \ {\rm sdepth}(S/J_G^{(k+2)})\leq {\rm sdepth}(S/J_G^{(k)})$$In fact, we can prove something stronger as follows.

\begin{thm} \label{unmix}
Let $I$ be an unmixed squarefree monomial ideal and assume that ${\rm ht}(I)=d$. Then for every integer $k\geq 1$ the inequalities $${\rm sdepth}(I^{(k+d)})\leq {\rm sdepth}(I^{{(k)}})$$and$${\rm sdepth}(S/I^{(k+d)})\leq {\rm sdepth}(S/I^{{(k)}})$$hold.
\end{thm}
\begin{proof}
Let $A= \{x_1, \ldots, x_n\}$ be the whole set of variables. Then for every prime ideal  $\frak{p}\in {\rm Ass}(S/I)$, we have $|\frak{p}\cap A|=d$. Hence the assertion follows from Remark \ref{phi}.
\end{proof}
Sine the cover ideal of every graph $G$ is unmixed of height two, we conclude the following result.
\begin{cor} \label{nonbi}
Let $G$ be an arbitrary graph and $J_G$ be the cover ideal of $G$. Then for every integer $k\geq 1$, the inequalities $${\rm sdepth}(J_G^{(k+2)})\leq {\rm sdepth}(J_G^{(k)})$$and $${\rm sdepth}(S/J_G^{(k+2)})\leq {\rm sdepth}(S/J_G^{(k)})$$hold.
\end{cor}

\begin{cor} \label{lim}
Let $I$ be an unmixed squarefree monomial ideal and assume that ${\rm ht}(I)=d$. Then for every integer $1\leq \ell \leq d$ the sequences $$\bigg\{{\rm sdepth}(S/I^{(kd+\ell)})\bigg\}_{k\in \mathbb{Z}_{\geq 0}} \ \ \ \ and  \ \ \ \  \bigg\{{\rm sdepth}(I^{(kd+\ell)})\bigg\}_{k\in \mathbb{Z}_{\geq 0}}$$ converge.
\end{cor}
\begin{proof}
Note that by Theorem \ref{unmix}, the sequences $$\bigg\{{\rm sdepth}(S/I^{(kd+\ell)})\bigg\}_{k\in \mathbb{Z}_{\geq 0}} \ \ \ \ {\rm and}  \ \ \ \  \bigg\{{\rm sdepth}(I^{(kd+\ell)})\bigg\}_{k\in \mathbb{Z}_{\geq 0}}$$ are both nonincreasing and so convergent.
\end{proof}

We do not know whether the Stanley depth of symbolic powers of a squarefree monomial ideal stabilizes. However, Corollary \ref{lim} shows that one can expect a nice limit behavior for the Stanley depth of symbolic powers of squarefree monomial ideals. Indeed it shows that for unmixed squarefree monomial ideals of height $d$, there exist two sets $L_1, L_2$ of cardinality $d$, such that $${\rm sdepth}(S/I^{(k)})\in L_1\ \ \  {\rm and}  \ \ \ \ {\rm sdepth}(I^{(k)})\in L_2,$$ for every $k\gg 0$. The following theorem shows that the situation is even better. Indeed we can even choose the sets $L_1$ and $L_2$ of smaller cardinality.

\begin{thm} \label{better}
Let $I$ be an unmixed squarefree monomial ideal and assume that ${\rm ht}(I)=d$. Suppose that $t$ is the number of positive divisors of $d$. Then
\begin{itemize}
\item[(i)] There exists a set $L_1$ of cardinality $t$, such that ${\rm sdepth}(S/I^{(k)})\in L_1$, for every $k\gg 0$.
\item[(ii)] There exists a set $L_2$ of cardinality $t$, such that ${\rm sdepth}(I^{(k)})\in L_2$, for every $k\gg 0$.
\end{itemize}
\end{thm}
\begin{proof}
(i) Based on Corollary \ref{lim}, it is enough to prove that for every couple of integers $1\leq \ell_1, \ell_2\leq d$, with ${\rm gcd}(d, \ell_1)=\ell_2$, we have $$\lim_{k\rightarrow \infty} {\rm sdepth}(S/I^{(kd+\ell_1)})= \lim_{k\rightarrow \infty}{\rm sdepth}(S/I^{(kd+\ell_2)}).$$Set $m=\frac{\ell_1}{\ell_2}$. Then by Corollary \ref{csymbol}, $$\lim_{k\rightarrow \infty}{\rm sdepth}(S/I^{(kd+\ell_2)})\geq \lim_{k\rightarrow \infty}{\rm sdepth}(S/I^{(mkd+m\ell_2)})=$$
$$\lim_{k\rightarrow \infty}{\rm sdepth}(S/I^{(mkd+\ell_1)})=\lim_{k\rightarrow \infty}{\rm sdepth}(S/I^{(kd+\ell_1)}),$$where the last equality holds, because the sequence $$\bigg\{{\rm sdepth}(S/I^{(mkd+\ell_1)})\bigg\}_{k\in \mathbb{Z}_{\geq 0}}$$ is a subsequence of the convergent sequence $$\bigg\{{\rm sdepth}(S/I^{(kd+\ell_1)})\bigg\}_{k\in \mathbb{Z}_{\geq 0}}.$$

On the other hand, since ${\rm gcd}(d, \ell_1)=\ell_2$, there exists an integer $m'\geq 1$, such that $m'\ell_1$ is congruence $\ell_2$ modulo $d$. Now by a similar argument as above, we have $$\lim_{k\rightarrow \infty}{\rm sdepth}(S/I^{(kd+\ell_1)})\geq \lim_{k\rightarrow \infty}{\rm sdepth}(S/I^{(m'kd+m'\ell_1)})=$$
$$\lim_{k\rightarrow \infty}{\rm sdepth}(S/I^{(kd+\ell_2)}),$$and hence $$\lim_{k\rightarrow \infty} {\rm sdepth}(S/I^{(kd+\ell_1)})= \lim_{k\rightarrow \infty}{\rm sdepth}(S/I^{(kd+\ell_2)}).$$

(ii) The proof is similar to the proof of (i).
\end{proof}





\end{document}